\documentclass{amsart}
\usepackage{amssymb}
\usepackage{enumitem}
\usepackage[color=green!40]{todonotes}
\usetikzlibrary{patterns}
\usepackage{mathtools}
\usepackage{xfrac}
\usepackage{comment}
\usepackage{tikz,pgfplots}
\usepackage{tikz-cd}
\usepackage{mathrsfs}
\usepackage{bbm}
\usepackage{hyperref}
\usepackage{environ}
\makeatletter
\providecommand{\env@tikzpicture@save@env}{}
\providecommand{\env@tikzpicture@process}{}

\theoremstyle{definition}
\newtheorem{remark}{Remark}[section]
\newtheorem{remarks}{Remarks}[section]

\theoremstyle{plain}
\newtheorem{theorem}{Theorem}[section]
\newtheorem{corollary}{Corollary}[section]
\newtheorem{lemma}{Lemma}[section]
\newtheorem{proposition}{Proposition}[section]

\theoremstyle{definition}
\newtheorem{definition}{Definition}[section]

\DeclareMathOperator{\N}{\mathbb{N}}
\DeclareMathOperator{\Z}{\mathbb{Z}}
\DeclareMathOperator{\Q}{\mathbb{Q}}
\DeclareMathOperator{\R}{\mathbb{R}}

\DeclareMathOperator{\St}{St}
\DeclareMathOperator{\C}{C}
\DeclareMathOperator{\Aff}{{\rm A}}

\newcommand{\2}{\mathbbm{2}}
\newcommand{\seq}{\subseteq}
\DeclareMathOperator{\Ker}{Ker}

\newcommand{\ba}{\mathbf{a}}
\newcommand{\bb}{\mathbf{b}}
\newcommand{\bc}{\mathbf{c}}
\newcommand{\bd}{\mathbf{d}}

\DeclareMathOperator{\conv}{conv}
\newcommand{\PolyZ}{{\sf PL}_{\mathbb{Z}}}
\newcommand{\Poly}{{\sf PL}}
\DeclareMathOperator{\MM}{\nabla_{\mathbb{Z}}}
\DeclareMathOperator{\den}{\rm den}

\renewcommand{\phi}{\varphi}
\renewcommand{\epsilon}{\varepsilon}

\renewcommand{\leq}{\leqslant}
\renewcommand{\geq}{\geqslant}

\DeclareMathOperator{\Rad}{Rad}
\DeclareMathOperator{\Max}{Max}

\newcommand{\MValg}{{\sf MV}}    		
\newcommand{\MV}{{\sf MV}}
\newcommand{\Au}{{{\sf A_1^\ell}}}      
\newcommand{\MValgs}{{\sf MV_{st}}}		
\newcommand{\MVfp}{{\sf MV_{fp}}}
\newcommand{\MVfpop}{{\sf MV_{fp}^{\rm op}}}

\newcommand{\As}{{{{\sf A_1}}}}         
\renewcommand{\P}{{\sf ES}}			    
\renewcommand{\S}{{\sf S}}			    

\newcommand{\Set}{{\sf Set}}
\newcommand{\Sett}{{\sf Set^2}}

\newcommand{\A}{\mathscr{A}}
\DeclareMathOperator{\Bd}{\mathscr{M}_{\rm b}}
\makeatletter
\@namedef{subjclassname@2020}{%
  \textup{2020} Mathematics Subject Classification}
\makeatother

\title[Two-sorted algebraic theory of states, and the universal states]{The two-sorted algebraic theory of states,\\ and the universal states of MV-algebras} 
\author{Tom\'a\v{s} Kroupa}
\address[T. Kroupa]{Institute of Information Theory and Automation of the CAS \\ Pod Vod\'arenskou v\v{e}\v{z}\'i 4, 182\,08 Prague, Czech Republic\\ \and Artificial Intelligence Center, Faculty of Electrical Engineering, Czech Technical University in Prague, Czech Republic}
\email{tomas.kroupa@fel.cvut.cz}
\author{Vincenzo Marra}
\address[V. Marra]{Dipartimento di Matematica ``Federigo Enriques'' \\ Universit\`a degli Studi di Milano \\ Via Cesare Saldini 50, 20133 Milano, Italy}
\email{vincenzo.marra@unimi.it}


\begin{document}
\begin{abstract}States of unital Abelian lattice-groups---that is, normalised positive group homomorphisms to $\R$---provide an abstraction of  expected-value operators. A well-known theorem due to Mundici asserts that the category of unital lattice-groups (with unit-preserving lattice-group homomorphisms as morphisms) is equivalent to the algebraic category of MV-algebras, and their homomorphisms. Through this equivalence, states of lattice-groups naturally correspond to certain $[0,1]$-valued functionals on MV-algebras, which are also known as states in the literature. In this paper we allow states to take values in any unital lattice-group (or in any MV-algebra, in the MV-algebraic setting) rather than just in $\R$ (or just in $[0,1]$, respectively).
  We introduce a~two-sorted algebraic theory whose models are precisely states of MV-algebras. We extend Mundici's equivalence to one between the category of MV-algebras with states as morphisms, and the category of unital Abelian lattice-groups  with, again, states as morphisms. Thus, the models of our two-sorted theory may also be regarded as states between unital Abelian lattice-groups, to within an~equivalence of categories. 
As our first main result, we derive the existence of the \emph{universal state} of any MV-algebra from the existence of free algebras in multi-sorted algebraic categories. The significance of the universal state of a~given algebra is that it provides (an algebraic abstraction of) the most general expected-value operator on that  algebra---a construct that is not available if one insists that states be real-valued. In the remaining part of the paper, we seek concrete representations of such universal states. We begin  by clarifying the relationship of  universal states with the theory of affine representations of lattice-groups: the universal state $A\to B$ of the MV-algebra $A$ is shown to coincide with a certain modification of Choquet's affine representation (of~the~unital lattice-group corresponding to $A$) if, and only if, $B$ is semisimple. An MV-algebra is locally finite if each one of its finitely generated subalgebras is finite; locally finite MV-algebras are semisimple, and Boolean algebras are instances of locally finite MV-algebras. Our second main result is then that the universal state of any locally finite MV-algebra has semisimple codomain, and can thus be described through our adaptation of Choquet's affine representation.
\end{abstract}

\maketitle

\section{Introduction}
We are concerned in this paper with the algebraic theory of states of MV-algebras and Abelian lattice-groups with a (strong order) unit, here called ``unital Abelian $\ell$-groups''. Usually, states are defined as normalised positive real-valued linear functionals
on Riesz spaces with unit or, more generally, normalised positive group homomorphisms to $\R$ of unital Abelian $\ell$-groups \cite{Goodearl86}.  Part of their importance stems from the~long-recognised fact  that unital Abelian lattice-groups and their states provide an abstraction of bounded real random variables and of expected-value operators, respectively. To illustrate, the collection $\Bd{(X,\A)}$ of bounded $\A$-measurable functions from $X$ to $\R$ is an Abelian $\ell$-group under pointwise addition and order; and the function $X\to \R$ constantly equal to $1$ is a unit of this $\ell$-group, due to our boundedness assumption. If $\mu$ is a probability measure on $\A$, then the expected-value operator $\int_X-\,{\rm d}\mu\colon \Bd{(X,\A)}\to \R$ is a state of the unital Abelian $\ell$-group  $\Bd{(X,\A)}$. 

In this paper we allow states to take values in any unital Abelian $\ell$-group, and not just in the real numbers; thus, if $G$ and $H$ are such groups, a state of $G$ with values in $H$ is a positive group homomorphism from $G$ to $H$ that carries the unit of $G$ to the unit of $H$. We shall see that this level of generality allows us to investigate universal constructions that, while  ubiquitous in algebra, are not available if one insists that states be real-valued. We shall be interested, specifically, in the existence of a ``most general state'', or \emph{universal state}, of a given unital Abelian $\ell$-group. We will prove that such a universal state indeed always exists, as a consequence of the standard fact that free algebras exist in algebraic categories. One obstruction to this plan is that the category of unital Abelian $\ell$-groups and their unit-preserving homomorphisms is not algebraic with respect to its underlying-set functor. This is  because the characteristic property of the unit $1$, that its multiples $n\coloneqq {1+\cdots+1}$ ($n$ times) should eventually exceed any given element, is not even  definable in first-order logic, by a standard compactness argument. Nonetheless, a well-known result of Mundici \cite{Mundici86} tells us that the category of unital Abelian $\ell$-groups is equivalent to the algebraic category of MV-algebras \cite{CignoliOttavianoMundici00} and their homomorphisms. We shall prove a version of Mundici's result for states, rather than just for unital homomorphisms, that will allow us to carry our programme out to within an equivalence of categories.

To begin with, the theory of (real-valued)~states has already been adapted to MV-algebras. The original reference for this is \cite{Mundici95}, where MV-algebraic states  were introduced with the  motivation of modelling the notion of ``average truth degree" in many-valued logic. For a primer on states of MV-algebras see \cite{Mundici11,FlaminioKroupa15}.
In line with what is discussed above for lattice-groups, in this paper we consider states between any two MV-algebras; see Definition \ref{d:state} below. Basic facts about the functor $\Gamma$ that features in Mundici's equivalence are recalled in Section \ref{sec:Gamma}. In Theorem \ref{thm:generalGamma} we extend Mundici's equivalence to one between  $\MValgs$ (the category whose objects are MV-algebras and whose morphisms are states) and $\As$ (the category whose objects are unital Abelian $\ell$-groups and whose morphisms are states). In Section \ref{s:statesvar} states are treated as two-sorted algebras, cf.\ Definition \ref{d:statesasalgs}. The elementary and yet key Proposition~\ref{p:equational} about their equational presentation is proved. The proposition enables us to identify the~category of states $\P$ with the category of models in $\Set$ of a~finitely axiomatised two-sorted equational theory. There is a  corresponding, non-algebraic category $\S$ of~states between unital Abelian $\ell$-groups. Theorem \ref{t:eqcharofstates} then shows that the categories $\P$ and $\S$ are equivalent. Section \ref{s:freestates} deals with free objects (free states) in $\P$, and with universal states (see below for precise definitions). We describe the free object generated by a two-sorted set  by universal states and binary coproducts in the category of MV-algebras (Theorem \ref{t:freesums}). This constitutes our first main result.

The second part of our paper is devoted to the issue of representing universal states explicitly, insofar as this is possible. For this, in Section \ref{s:choquet} we show how Choquet's theory of  affine representations \cite[Chapters 5--7]{Goodearl86} relates to the construction of a~universal state. Specifically, Proposition~\ref{pro:semisimplecodomain} says that the codomain of a~universal state is Archimedean (or semisimple, in the case of MV-algebras) precisely when it coincides with the  extended form of the affine representation that we introduce. Thus, universal states with a semisimple codomain admit of a satisfactory concrete description through affine representations. Our second main result, proved in Section \ref{s:univlocfin}, is that the codomain of~the~universal state of~any locally finite MV-algebra (in particular, of any Boolean algebra) is semisimple. The~proof uses the~duality between finitely presented MV-algebras and the category of compact rational polyhedra with piecewise linear maps with integer coefficients as morphisms.

We assume familiarity with MV-algebras and unital Abelian $\ell$-groups; see \cite{CignoliOttavianoMundici00, Mundici11} and \cite{BigardKeimelWolfenstein77, Goodearl86} for background information. We  often adopt the~standard practice in algebra of omitting underlying-set functors, if clarity is not impaired. We shall also omit parentheses in application of functors and functions, writing e.g.\ $FI$ in place of $F(I)$, when this improves readability. We assume $\N\coloneqq \{1,2,\ldots\}$.

\section{Mundici's equivalence, for states}\label{sec:Gamma} Let $\Au$ be the category that has unital Abelian $\ell$-groups as objects and unital $\ell$-homomorphisms as morphisms, and let $\MValg$ be the category of MV-algebras and their homomorphisms. In \cite[Theorem 3.9]{Mundici86}, Mundici established a categorical equivalence between $\Au$ and  $\MValg$ which we recall here, without proofs. To each unital $\ell$-group $G$ we associate its \emph{unit interval}
\begin{equation}\label{def:Gamma}
 \Gamma(G,1)\coloneqq \{a\in G\mid 0\leq a \leq 1\},
 \end{equation}
equipped with the operations
\begin{align*}
a\oplus b & \coloneqq (a+b)\wedge 1, \\
\neg a & \coloneqq 1-a.
\end{align*}
Then $(\Gamma(G,1),\oplus,\neg,0)$ is an MV-algebra \cite[Proposition 2.1.2]{CignoliOttavianoMundici00}. We will also make use of binary operations $\odot$ and $\ominus$ on $\Gamma(G,1)$ defined by
\begin{align*}
  a\odot b & \coloneqq \neg(\neg a \oplus \neg b),\\ 
  a\ominus  b & \coloneqq a \odot \neg b.
  \end{align*}
The operations $+,\oplus$, and $\odot$ in $\Gamma(G,1)$ are related as follows (see \cite[Lemma~2.1.3(i)]{CignoliOttavianoMundici00}):
\begin{equation}\label{eq:sum}
  a+b=(a\oplus b) + (a\odot b), \qquad a,b\in \Gamma(G,1).
\end{equation}

Since every unital $\ell$-homomorphism $f\colon G\to H$ restricts to a homomorphism of MV-algebras $\Gamma(G,1)\to \Gamma(H,1)$, we obtain a functor
\begin{equation}\label{eq:Gamma}
\Gamma\colon \Au \longrightarrow \MValg.
\end{equation}
If the unit $1$ is understood, we  write $\Gamma G$ in place of $\Gamma(G,1)$.

In order to describe a functor in the other direction, the notion of good sequence was introduced in \cite{Mundici86} (see also \cite[Chapter 2 and 7]{CignoliOttavianoMundici00}). For $M$ an MV-algebra, we say that $\ba\coloneqq (a_{i})_{i\in \N} \in M^{\N}$ is a \emph{good sequence} in $M$ if $a_{i}\oplus a_{i+1}=a_{i}$ for each $i\in \N$, and there is $n_{0}\in \N$ such that $a_{n}=0$ for all $n\geq n_{0}$. We shall write $(a_{1},\dots,a_{k})$ in place of $(a_{1},\dots,a_{k},0,0,\dots)$;
in particular, $(a)$ is short for $(a,0,0,\dots)$, given $a\in M$. Addition of good sequences $(a_{i})_{i\in \N}$ and $(b_{i})_{i\in \N}$ is defined by
\[
	(a_{i})_{i\in \N} + (b_{i})_{i\in \N} \coloneqq (a_{i}\oplus(a_{i-1}\odot b_{1  }) \oplus \dots \oplus (a_{1}\odot b_{i-1})\oplus b_{i})_{i\in \N}.
\]
The set of all good sequences in $M$ equipped with $+$ becomes a commutative monoid $A_{M}\subseteq M^{\N}$ with neutral element $(0)$. By general algebra, the full inclusion of the category of Abelian groups into that of commutative monoids has a left adjoint; write  $\eta_{A_M}\colon A_M\to \Xi M$ for the component at $A_M$ of the unit of this adjunction. The monoid $A_{M}$ can be shown to be cancellative, so the monoid  homomorphism $\eta_{A_M}$ is injective. To describe the elements of $\Xi M$ explicitly, let us say two ordered pairs of good sequences $(\ba,\bb)$ and $(\ba',\bb')$ are \emph{equivalent} if 
\[
\ba+\bb'=\ba'+ \bb,
\]
and let us write $[\ba,\bb]$ for the equivalence class of $(\ba,\bb)$. Then $\Xi M$ is defined as the set of all equivalence classes of the form $[\ba,\bb]$ equipped with  the addition
\[
[\ba,\bb] + [\bc,\bd] \coloneqq [\ba+\bc,\bb+\bd],
\]
with the neutral element $[(0),(0)]$, and  with the unary inverse operation
\[
-[\ba,\bb]\coloneqq[\bb,\ba].
\]
  Moreover, the monoid $A_M$ is lattice-ordered by the restriction of  the product order of $M^{\N}$, and this lattice order on $A_{M}$ extends (in the obvious sense, through the injection $\eta_{A_M}$) to exactly one  translation-invariant lattice order on $\Xi M$. Thus, $\Xi M$ is an Abelian $\ell$-group with unit $[(1),(0)]$. 
\begin{lemma}\label{thm:review}
For any  MV-algebra $M$, and any unital Abelian $\ell$-group  $G$, the following hold.
\begin{enumerate}
\item The function $\phi_M \colon M \to \Gamma\Xi M$ given by \[\phi_M(a)\coloneqq [(a),(0)],\quad a\in M,\] is an isomorphism of MV-algebras.
\item The lattice-ordered monoids $G^{+}\coloneqq\{a\in G\mid a\geq 0\}$ and $A_{\Gamma G}$ are isomorphic through the function $g\colon G^{+}\to A_{\Gamma G}$ that sends $a\in G^{+}$ to the unique good sequence \[g(a)\coloneqq (a_{1},\dots,a_{n})\] of elements $a_{i}\in \Gamma G$ such that $a=a_{1}+\dotsb + a_{n}$.
\item The function $\epsilon_G\colon G\to \Xi\Gamma G$ defined by
\begin{equation*}
\epsilon_G (a) \coloneqq [g(a^+),g(a^-)], \quad a\in G,
\end{equation*}
is an isomorphism of  unital $\ell$-groups, where $g$ is as in item (2) above, and, as usual, $a^{+}\coloneqq a\vee 0$, $a^{-}\coloneqq-a \vee 0$.
\end{enumerate}
\end{lemma}
\begin{proof}
See Theorem 2.4.5, Lemma 7.1.5, and Corollary 7.1.6 in \cite{CignoliOttavianoMundici00}.
\end{proof}
A  homomorphism $h\colon M\to N$ of MV-algebras lifts to a function $h^{*}\colon A_{M}\to A_{N}$ upon setting $h^{*}((a_{i})_{i\in \N})\coloneqq(h(a_{i}))_{i\in \N}$. Then Lemma \ref{thm:review} and the universal construction of $\Xi M$ from $A_M$ entail that $h^{*}$ has exactly one extension to a unital $\ell$-homomorphism $\Xi h\colon \Xi M\to \Xi N$. We thereby obtain a functor 
\begin{equation}\label{eq:Xi}
\Xi\colon\MValg\longrightarrow \Au.
\end{equation}
\begin{theorem}[Mundici's equivalence]\label{thm:Mundiciequiv}
The   functors $\Gamma\colon\Au \rightarrow \MValg$ and $\Xi\colon\MValg\rightarrow \Au$  form an equivalence of categories. 
\end{theorem}
\begin{remark}\label{r:abuseoflang}In what follows we shall often tacitly identify $M$ with  $\Gamma\Xi M\subseteq \Xi M$, and thus speak of functions defined on $M$ having an extension to $\Xi M$, etc.  
\end{remark}
In the rest of this section we lift the equivalence of Theorem \ref{thm:Mundiciequiv} from homomorphisms to states. 
\begin{definition}[States]\label{d:state}For MV-algebras $M$ and $N$, a function $s\colon M\to N$ is a \emph{state} (\emph{of $M$ with values in $N$}) if $s(1)=1$, and for each $a,b\in M$ with $a \odot b=0$ the equality $s(a\oplus b) =s(a)+s(b)$ holds, where $+$ is interpreted in $\Xi N$. For unital Abelian $\ell$-groups $G$ and $H$, a function $s\colon G\to H$ is a \emph{state} (\emph{of $G$ with values in $H$}) if $s(1)=1$, and $s$ is a group homomorphism that is \emph{positive}, i.e., for each $g\in G^+$ we have $s(g)\in H^+$. We write  $\MValgs$ for the category whose objects are MV-algebras and whose morphisms are states, and $\As$ for the category whose objects are  unital Abelian $\ell$-groups and whose morphisms are states.
\end{definition}

States of MV-algebras can be defined in several equivalent ways. For instance, a~function $s\colon M\to N$ satisfying $s(1)=1$ is a state if, and only if, for each $a,b\in M$ with $a \odot b=0$, the equalities $s(a\oplus b) =s(a)\oplus s(b)$ and $s(a)\odot s(b)=0$ hold---see \cite[Proposition 3.3]{Kroupa2018}. In Section \ref{s:statesvar} we will give an equational characterisation of states (Proposition \ref{p:equational}).

\begin{remarks}\label{r:states} (1)\ It is elementary that  requiring the group homomorphism $s\colon G\to H$ to be positive is equivalent to asking that it be order-preserving.

\smallskip \noindent (2)\  States are a classical notion in the theory of partially ordered Abelian groups (see e.g.\ \cite{Goodearl86}), where they are most often assumed to have codomain $\R$. For emphasis, we  refer to the latter states as \emph{real-valued}. A significant example is the Lebesgue integral $\int_{[0,1]}- \,{\rm d}\lambda$ over the unital Abelian $\ell$-group of all continuous functions $[0,1]\to\R$. Let us stress that states do not necessarily preserve infima and suprema and, therefore, in general they are not  morphisms in $\Au$. 

\smallskip \noindent (3)\ States of MV-algebras also are a well-studied notion (see e.g.\ \cite{Mundici11,FlaminioKroupa15}), and they are usually  assumed to have codomain $[0,1]\subseteq \R$. We  refer to the latter states  as \emph{real-valued}. Analogously to the previous item, states of MV-algebras are not morphisms in $\MV$, in general.
\end{remarks}
\begin{lemma}\label{l:extfunctor}Let $M$ and $N$ be MV-algebras, and $G$ and $H$ be unital Abelian $\ell$-groups.
\begin{enumerate}
\item Any state of unital Abelian $\ell$-groups $s\colon G\to H$ restricts to a function $\Gamma(s)\colon\Gamma G\to\Gamma H$ that is a state of MV-algebras.
\item Any state of MV-algebras $s\colon M\to N$ has exactly one extension to a state of unital Abelian $\ell$-groups $\Xi(s)\colon \Xi M\to\Xi N$.
\end{enumerate}
\end{lemma}
\begin{proof}To prove the first item, observe that $0\leq a \leq 1$ in $G$ entails $0\leq s(a)\leq 1$ in $H$, because $s$ preserves the order; further, $s(1)=1$ by definition. If $0\leq  a_1,a_2\leq 1$ in $G$, and $a_1\odot a_2 =0$ in the MV-algebra $\Gamma G$, then (\ref{eq:sum}) yields $a_1\oplus a_2=a_2+a_2$.

For the second item, let us regard $s$ as a function $\Gamma\Xi M\to \Gamma\Xi N$. Then $s$ is easily seen to be order-preserving, \cite[Proposition 10.2]{Mundici11}. For $a,b\in  \Gamma\Xi M$ we have $s(a+b)=s(a)+s(b)$ as soon as $a+b$ belongs to $\Gamma\Xi M$. Indeed, the latter happens precisely when $a\odot b=0$ by (\ref{eq:sum}), and then $s(a+b)=s(a\oplus b)=s(a)+s(b)$ by the definition of state. Now existence and uniqueness of $\Xi(s)$ is granted by the general extension result \cite[Proposition 1.5]{HandelmanHiggsLawrence1980}. (The hypotheses of the cited proposition require  $G$ to be directed, and to have the Riesz interpolation property; it is classical and elementary that any $\ell$-group satisfies these properties.) 
\end{proof}
In light of Lemma \ref{l:extfunctor}  we consider the functors
\[
\Gamma\colon \As\longrightarrow \MValgs
\]
and
\[
\Xi\colon \MValgs\longrightarrow\As
\]
that extend the homonymous ones in \eqref{eq:Gamma} and \eqref{eq:Xi}, respectively.

 \begin{theorem}[Mundici's equivalence, for states]\label{thm:generalGamma}
The functors $\Gamma\colon \As\to \MValgs$ and $\Xi\colon \MValgs\longrightarrow\As$ 
 form an equivalence of categories.
 \end{theorem}
 \begin{proof}
The isomorphism $\phi_M \colon M \to \Gamma\Xi M$  from Lemma \ref{thm:review}(1) is natural: if  $s\colon M\to N$ is a state of MV-algebras, the naturality square 
\[
\begin{tikzcd}
M \ar[rr, "\phi_M"]    \ar[d,   "s"'] & {} &  \Gamma\Xi M \ar[d,   "\Gamma\Xi(s)"]\\
N \ar[rr,  "\phi_N"'] & {}    & \Gamma\Xi N
\end{tikzcd}
\]
commutes by the definitions of $\Xi$ and $\Gamma$.

The isomorphism $\epsilon_{G}\colon G\to \Xi\Gamma G$ is also natural: if  $s\colon G\to H$ is a state of unital Abelian $\ell$-groups, the naturality square 
\[
\begin{tikzcd}
G \ar[rr, "\epsilon_G"]    \ar[d,   "s"'] & {} &  \Xi\Gamma G \ar[d,   "\Xi\Gamma(s)"]\\
H \ar[rr,  "\epsilon_H"'] & {}    & \Xi\Gamma H
\end{tikzcd}
\]
commutes. Indeed, $\Xi\Gamma(s)\epsilon_G$ and $\epsilon_Hs$ are states $G\to \Xi\Gamma H$, and it follows from direct inspection of the definitions involved that they  agree on the unit interval of~$G$; but then they agree on the whole of $G$, by (1) and (2) in Lemma~\ref{l:extfunctor}.
 \end{proof}	
 \section{The two-sorted variety of states}\label{s:statesvar}
For  multi-sorted universal algebra see the pioneering \cite{Higgins63, BirkhoffLipson70},  and the textbook reference \cite{AdamekRosickyVitale11}. We are  concerned with the two-sorted case only. Unlike \cite{BirkhoffLipson70}, and like \cite{AdamekRosickyVitale11}, we allow arbitrary multi-sorted sets as carriers of algebras---no non-emptyness requirement is enforced. The difference is immaterial for the present paper, because each of our two sorts has constants. We recall that the product  category $\Sett\coloneqq \Set\times \Set$ of two-sorted sets and two-sorted functions has as objects the ordered pairs $(A,B)$ of sets, and as  morphisms 
\[
f\colon (A_{1},B_{1}) \longrightarrow (A_{2},B_2)
\]
the ordered pairs $f\coloneqq (f_{1},f_{2})$ of functions
\begin{align*}
f_{1}& \colon A_{1} \to A_{2},\\
f_{2}&\colon B_{1} \to B_{2}.
\end{align*}
Composition of morphisms and  identity morphisms are defined componentwise. 
We  consider two  sorts $\mathcal{R}$ and $\mathcal{E}$ of \emph{random variables} and \emph{degrees of expectation}, respectively, and operations as follows.
\begin{itemize}
\item[(T1)] Operations $\oplus\colon \mathcal{R}^2\to\mathcal{R}$, $\neg\colon \mathcal{R}\to \mathcal{R}$, and $0\colon \mathcal{R}^{\emptyset}\to\mathcal{R}$. Thus, these are operations of arities $2$, $1$, and $0$, respectively,  in the  sort of random variables.
\item[(T2)] Operations $\oplus\colon \mathcal{E}^2\to\mathcal{E}$, $\neg\colon \mathcal{E}\to \mathcal{E}$, and $0\colon \mathcal{E}^{\emptyset}\to\mathcal{E}$. Thus, these are operations of arities $2$, $1$, and $0$, respectively,  in the  sort of expectation degrees. They are purposefully denoted by the same symbols as their counterparts in the sort $\mathcal{R}$.
\item[(T3)] One operation $s\colon \mathcal{R}\to \mathcal{E}$ from the sort of random variables to that of degrees of expectation.
\end{itemize}
Items (T1)--(T3) define a two-sorted (similarity) type. For the sake of clarity, let us spell out that a two-sorted function $(m,n)\colon (M_{1},N_{1})\to (M_{2},N_{2})$ is a homomorphism between algebras of this two-sorted  type precisely when
$m\colon M_{1}\to M_{2}$ and $n\colon N_{1}\to N_{2}$ are homomorphisms in the  type of $\mathcal{R}$ and $\mathcal{E}$, respectively, and moreover the square
\begin{equation}\label{eq:diagmn}
\begin{tikzcd}
 M_{1} \ar[rr, "s"]    \ar[d,   "m"'] & {} & N_{1}  \ar[d,   "n"]\\
M_{2} \ar[rr,  "s"'] & {}    & N_{2}
\end{tikzcd}
\end{equation}
commutes.
\begin{definition}[States as two-sorted algebras]\label{d:statesasalgs}
A \emph{state} is an algebra $(M,N)$ of the two-sorted  type (T1)--(T3) such that the following equational conditions hold.
\begin{itemize}
\item[(S1)] $(M,\oplus,\neg,0)$ is an MV-algebra.
\item[(S2)] $(N,\oplus,\neg,0)$ is an MV-algebra.
\item[(S3)] For every $a,b\in M$,  $s\colon M\to N$  satisfies  
\begin{enumerate}[label=(A\arabic*)]
\item\label{A1} $s(a\oplus b)=s(a) \oplus s(b\wedge \neg a)$,
\item\label{A2} $s(\neg a)=\neg s(a)$, and
\item\label{A3} $s(1)=1$.
\end{enumerate}
\end{itemize} 
\end{definition}
\begin{remark}
  The presented axiomatisation (S3) is originally inspired by that of internal states \cite{FlaminioMontagna09}. It was used already in \cite{KroupaMarraSC17} in case of states whose domains are Boolean algebras; see also \cite{Kroupa2018}.
\end{remark}

\noindent Let us emphasise that in the above and throughout we denote a two-sorted algebra simply by its underlying two-sorted set $(M,N)$. All operations---$\oplus$, $s$, and so forth---are tacitly understood. This is in keeping with standard usage in algebra.

 Item (S3) in   Definition \ref{d:statesasalgs} amounts to requiring that  $s\colon M\to N$ be a state:
\begin{proposition}\label{p:equational}Let $s\colon M \to N$ be a function between MV-algebras $M$ and $N$. The following are equivalent.
\begin{enumerate}
\item The function $s$ is a state.
\item  The function $s$ satisfies (S3) in Definition \ref{d:statesasalgs}:

\begin{enumerate}[label=(A\arabic*)]
\item\label{A1} $s(a\oplus b)=s(a) \oplus s(b\wedge \neg a)$,
\item\label{A2} $s(\neg a)=\neg s(a)$,
\item\label{A3} $s(1)=1$.
\end{enumerate}
\end{enumerate}
\end{proposition}

\begin{proof}
We will need the following equations, valid in all MV-algebras: 
\begin{enumerate}[label=(MV\arabic*)]
  \item \label{MV1} $b\wedge \neg a=b\ominus (a\odot b)$,
  \item \label{MV2}	$a\oplus b=a\oplus (b\wedge \neg a)$,
  \item\label{MV3}	$a\odot (b\wedge \neg a)=0$.
  \end{enumerate}  
In detail, \ref{MV1} holds by the definition of $\wedge$, since
  \[
  b\wedge \neg a =b\odot (\neg b \oplus \neg a)=b\odot \neg (a\odot b)=b\ominus (a\odot b).
  \]
  The operation~$\oplus$ distributes over $\wedge$ by \cite[Proposition 1.1.6]{CignoliOttavianoMundici00}, so
  \[
  a\oplus (b\wedge \neg a) = (a\oplus b) \wedge (a\oplus \neg a)=a\oplus b,
  \]
  which proves \ref{MV2}. Finally, for \ref{MV3},
  \[
  a\odot (b\wedge \neg a)=a\odot (b\ominus (a\odot b))=(a\odot b) \odot \neg (a\odot b)=0.
  \]

  Assume $s$ is a state. Then \ref{A3} holds by definition. Further, \ref{MV2}, \ref{MV3}, and the definition of state yield
  \[
    s(a\oplus b)=s(a\oplus (b\wedge \neg a))=s(a) + s(b\wedge \neg a).
  \]
  Then $s(a) + s(b\wedge \neg a)\in N$, and so the  $+$ above agrees in fact with $\oplus$ by (\ref{eq:sum}). Therefore, \ref{A1} holds. Finally, since $a\odot \neg a =0$, we can write
  \[
    1=s(a\oplus \neg a)=s(a) + s(\neg a),
  \]
  which proves \ref{A2} because $1-s(a)$ equals $\neg s(a)$ in $\Xi N$ . Thus, $s$ satisfies \ref{A1}--\ref{A3}.

  Conversely, assume   $s\colon M \to N$ has properties \ref{A1}--\ref{A3}. We first prove that $s$ is order-preserving. Assume $a\leq b$, or equivalently, by definition, $b=a\oplus (b\ominus a)$. Then \ref{A1} yields
  \[
    s(b)=s(a\oplus (b\ominus a))=s(a)\oplus s((b\ominus a) \wedge \neg a) \geq s(a),
 \]
 where the last inequality follows from monotonicity of $\oplus$ in each coordinate \cite[Lemma 1.1.4]{CignoliOttavianoMundici00}.

  Let $a\odot b=0$. Then \ref{MV1} gives $b\wedge \neg a=b$ and, by \ref{A1},
  \[
    s(a\oplus b)=s(a)\oplus s(b).\]
    It remains to check that $s(a)\odot s(b)=0$, which implies $s(a)\oplus s(b)=s(a)+s(b)$ as is to be shown. Employing \ref{A1}, \ref{A3}, and the assumption $a\odot b=0$, we get
    \[
     s(\neg a)\oplus s(\neg b \wedge a)= s(\neg a \oplus \neg b)=s(\neg (a\odot b))=s(\neg  0)=1.
    \]
    Since $s$ is order-preserving, and since $\oplus$ is monotone in each coordinate by \cite[Lemma 1.1.4]{CignoliOttavianoMundici00}, we infer $s(\neg a) \oplus s(\neg b)\geq s(\neg a)\oplus s(\neg b \wedge a)=1$. Then by \ref{A2} this gives
    \[
      s(a)\odot s(b)=\neg (s(\neg a) \oplus s(\neg b))=0,
    \]
    which completes the proof.
\end{proof}

We consider the category $\P$ (for ``Equational States'') of states in the sense of Definition \ref{d:statesasalgs}, and their homomorphisms. By its very definition, $\P$ is the category of models in $\Set$ of a (finitely axiomatised) two-sorted equational theory. Thus, $\P$ is a two-sorted variety---i.e., it is closed under homomorphic images, subalgebras, and products inside the category of all algebras of the type (T1)--(T3)---by the easy implication in Birkhoff's Variety Theorem. (See \cite{AdamekRosickyVitale12} for details on Birkhoff's Theorem in the multi-sorted setting.) Let us recall that ``homomorphic images'' here are the codomains of those homomorphisms that are surjective in each sort, and that these are exactly the regular epimorphisms \cite[Corollary 3.5]{AdamekRosickyVitale11}. 

 We further consider the category $\S$ of states whose objects are all states $G\to H$  of unital Abelian $\ell$-groups $G$ with values in any unital Abelian $\ell$-groups $H$, and whose morphisms are pairs of unital $\ell$-homomorphisms $G_1\to G_2$ and $H_1\to H_2$ making the obvious square commute. Thus, the objects of $\S$ are exactly the arrows $G\to H$ in $\As$.  We define a functor
\[\Gamma^2\colon\S\longrightarrow \P\]
by setting $\Gamma^2(G\to H)\coloneqq (\Gamma G,\Gamma H)$, for  $G\to H$ an object of $\S$, where the operation $s\colon \Gamma G\to\Gamma H$ of the two-sorted algebra $(\Gamma G,\Gamma H)$  is defined as the restriction of the given state $G\to H$ (cf.\ item (1) in Lemma \ref{l:extfunctor}). Concerning arrows, given the morphism $(g,h)\colon (G_1\to H_1)\to (G_2\to H_2)$ in $\S$---i.e., given the pair of unital $\ell$-homomorphisms $g\colon G_1\to G_2$ and $h\colon H_1\to H_2$ that form the required commutative square with the states $G_i\to H_i$, $i=1,2$---we set $\Gamma^2(g,h)\coloneqq (\Gamma g,\Gamma h)$, which is evidently  a morphism in $\P$. We also define a functor
\[\Xi^2\colon\P\longrightarrow \S\]
by setting $\Xi^2(M,N)$ to be the state $\Xi s\colon \Xi M \to \Xi N$ (cf.\ item (2) in Lemma \ref{l:extfunctor}), where $s\colon M\to N$ is the two-sorted operation of $(M,N)$. For a  morphism $(m,n)$ from $(M_1,N_1)$ to $(M_2,N_2)$, we let $\Xi^2(m,n)\coloneqq (\Xi m,\Xi n)$; this is a morphism in $\S$ because $\Xi$ is a functor. 
\begin{theorem}[Equational characterisation of states]\label{t:eqcharofstates}
	The functors  $\Gamma^2\colon\S\rightarrow \P$ and $\Xi^2\colon\P\rightarrow\S$ form an equivalence of categories.
\end{theorem}
\begin{proof}The morphism $(\phi_{M},\phi_{N})\colon (M,N) \to \Gamma^2\Xi^2(M,N)$  is an isomorphism in $\P$, where  the components of $(\phi_{M},\phi_{N})$ are as in Lemma \ref{thm:review} and thus are isomorphisms; naturality is verified componentwise, and thus follows at once from Theorem \ref{thm:generalGamma}.
Similarly, for every state $G\to H$ in $\As$,  $(\epsilon_{G},\epsilon_{H})\colon (G\to H)\to \Xi^2\Gamma^2(G,H)$ is an~isomorphism, where $\epsilon_{G}$ and $\epsilon_{H}$ are as in Lemma \ref{thm:review}, and naturality reduces to naturality in each sort, which  holds by Theorem \ref{thm:generalGamma}.
\end{proof}
\section{Free and universal states}\label{s:freestates}
For a set $I$, let us write $FI$ for the free MV-algebra generated by $I$, and  
\begin{align}\label{eq:freeMVunit}
\iota_I\colon I\to F I
\end{align}
for the ``inclusion of free generators'', i.e., for the component at $I$ of the unit of the free/underlying-set adjunction $F\dashv |-|$. 

By general algebraic considerations, the  functor
\[
|-|_2\colon\P \longrightarrow\Sett
\]
that takes a state of MV-algebras to its carrier two-sorted set has a left adjoint
\[
F^2\colon \Sett \longrightarrow \P .
\]
The existence of this  left adjoint has nothing to do with MV-algebras specifically, and is entailed by the existence of free algebras in varieties of multi-sorted algebras. The original reference for this latter result seems to be \cite[Section 5]{Higgins63}, with a~generalisation established in  \cite[Section 7]{BirkhoffLipson70}.
For a two-sorted set $S$, write $\eta_ S\colon S \to F^2S$ for the component at $S$ of the unit of the adjunction $F^2\dashv |-|_2$.   Then $\eta_S$ is characterised as the essentially unique two-sorted function $S \to F^2S$ such that, for any two-sorted function $f\colon S \to (M,N)$ with a state in the codomain, there is exactly one morphism $h\colon F^2S\to (M,N)$ in $\P$ making the diagram 
\[
\begin{tikzcd}
S \ar[r, "\eta_S"]    \ar[dr,   "f"'] &  F^2S \ar[d,  dashed,  "h"]\\
 &  (M,N) 
\end{tikzcd}
\]
commute. In algebraic parlance, $F^2S$ is a  state freely generated by the two-sorted set $S$; it is evidently unique to within a unique isomorphism, and therefore  any state satisfying the preceding  universal property will  be called \emph{the free state generated by~$S$}.

\begin{lemma}\label{l:freeS1} Consider the two-sorted set $(\emptyset,S_2)$, where $S_2$ is any set. Writing  $\2$ for the initial MV-algebra (=two-element Boolean algebra),  $(\2 ,FS_{2})$ is the free state in $\P$ generated by $(\emptyset,S_2)$, where  the operation $s$ is the unique homomorphism $\2\to FS_2$.
\end{lemma}
\begin{proof}Given  a two-sorted function $(f_!,f_2)\colon(\emptyset,S_2)\to (M,N)$, where $f_!\colon\emptyset \to M$ is the unique function from $\emptyset$ to $M$, write   $h_!\colon\2\to M$ for the unique homomorphism, and $h_2\colon FS_2\to N$ for the unique homomorphism  such that $f_2=h_2\iota_{S_2}$, where $\iota_{S_2}$ is as in (\ref{eq:freeMVunit}). Further, consider the component-wise inclusion  $\iota\colon (\emptyset,S_2)\subseteq (\2,FS_2)$. Then $(h_!,h_2)\colon (\2,FS_{2})\to (M,N)$ is a morphism in $\P$, because $\2$ is initial in $\MValgs$. Also, $(f_!,f_2)=(h_!,h_2)\iota$, and $(h_!,h_2)$ is clearly unique with this property.
\end{proof}
If $M$ is any MV-algebra, a state
\[
\upsilon_M\colon M\longrightarrow \Upsilon M
\]
is said to be universal (for $M$) if for each state $s\colon M\to N$ there is exactly one homomorphism of MV-algebras $h\colon \Upsilon M\to N$ satisfying $h\upsilon_M=s$. Universal states, when they exist, are evidently unique to within a unique isomorphism. Any state satisfying the preceding universal property will therefore be called \emph{the universal state} (\emph{of $M$}).
\begin{lemma}\label{l:freeS2}Consider the two-sorted set $(S_1,\emptyset)$, where $S_1$ is any set.
Then the free state $F^2(S_1,\emptyset)$ generated by $(S_1,\emptyset)$ is (isomorphic in $\P$ to) 
\[
(FS_{1},\Upsilon F S_1),
\]
where the operation $s$ is the universal state $\upsilon_{FS_1}$ of $FS_1$.
\end{lemma} 
\begin{proof}Let us display the components of $F^2(S_1,\emptyset)$ as the pair of MV-algebras $(A,B)$. If $M$ is any MV-algebra and $\mathbbm{1}$ is the terminal (=one-element) MV-algebra, we  consider the object $(M,\mathbbm{1})$ of $\P$ whose operation $s$ is the unique homomorphism $M\to \mathbbm{1}$.  Given any function $f\colon S_1\to M$, the unique two-sorted function $(f,!)\colon (S_1,\emptyset)\to (M,\mathbbm{1})$ whose first component is $f$ (and whose second component, necessarily, is the only possible function $!\colon\emptyset \to \mathbbm{1}$) has exactly one extension to a homomorphism $(h_1,!)\colon (A,B)\to (M,\mathbbm{1})$. The second component of this homomorphism is the only possible one $!\colon B\to \mathbbm{1}$. Then $h_1\colon A\to M$ is a homomorphism that extends $f$, and it is the unique such: if $h_1'\colon A\to M$ extends $f$ then $(h_1',!)\colon (A,B)\to(M,\mathbbm{1})$ extends $(f,!)$, and thus $h_1'=h_1$ by the uniqueness of~$(h_1,!)$. This shows that $A$ is $FS_1$. In the rest of this proof we write $FS_1$ and drop~$A$.

Let us next consider the operation $s\colon FS_1\to B$ of the object $(FS_1,B)$, with the intent of proving that the state $s$ is universal for $FS_1$. Consider any state $t\colon FS_1\to N$, and the corresponding object $(FS_1,N)$ of $\P$. Writing $\iota\colon S_1\to FS_1$ for the insertion of free generators, we consider the two-sorted function $(\iota,!)\colon (S_1,\emptyset)\to (FS_1,N)$. The universal property of $(FS_1,B)$ yields exactly one homomorphism $(1_{FS_1},h_2)\colon (FS_1,B)\to (FS_1,N)$ extending $(\iota,!)$, the first component of which is, necessarily, the identity $1_{FS_1}\colon FS_1\to FS_1$. The second component $h_2$ therefore satisfies $h_2s=t$. Further, $h_2$ is the only homomorphism with this property. Indeed, if $h'_2\colon B \to N$ satisfies $h'_2s=t$, then $(1_{FS_1},h'_2)\colon (FS_1,B)\to (FS_1,N)$ extends $(\iota,!)$ and so $h'_2=h_2$ by the universal property of $(FS_1,B)$. In conclusion, $s\colon FS_1\to B$ is the universal  state $\upsilon_{F S_1}$ of $FS_1$, and therefore $B$ is $\Upsilon F S_1$. This completes the proof.
\end{proof}
Lemma \ref{l:freeS1} establishes the existence of the universal state of a free MV-algebra. For the general case, we take a quotient. If $f\colon A\to B$ is a function, we write $\Ker{f}\coloneqq\{(x,y)\in A^2\mid fx=fy\}$. This notation is opposed to $\ker{f}\coloneqq f^{-1}\{0\}$, which we also use later on in the paper and which makes sense when $A$ and $B$ are either MV-algebras or lattice-groups. 

\begin{corollary}\label{c:existenceuniv}Every MV-algebra has a universal state.
\end{corollary}
\begin{proof}Let $M$ be any MV-algebra, write $FM$ for the free MV-algebra generated by the set  $M$, and let $q_1\colon FM\to M$ be the unique, automatically surjective homomorphism extending the identity function $M\to M$. By   Lemma \ref{l:freeS2}, $F^2(M,\emptyset)$ is $
(FM,\Upsilon F M)$, 
where the operation $s$ is the universal state $\upsilon_{FM}$ of $FM$. In the rest of this proof we write $\upsilon$, {\it tout court}.
 Set $\theta_1=\Ker{q_1}$ and $\theta_2=\langle\{(\upsilon x, \upsilon y) \mid (x,y)\in \theta_1\}\rangle$, where $\langle - \rangle$ denotes the congruence generated by $-$ on $\Upsilon F M$. Then $(\theta_1,\theta_2)$ is evidently a congruence on $(FM,\Upsilon F M)$, and it is the congruence generated by the  two-sorted relation $(\theta_1,\emptyset)$---indeed, any congruence on  $(FM,\Upsilon F M)$ that contains $(\theta_1,\emptyset)$ must contain $(\theta_1,\{(\upsilon x, \upsilon y) \mid (x,y)\in \theta_1\})$, by the compatibility with $\upsilon$, and thus must contain $(\theta_1,\theta_2)$.
 
 We consider the quotient state $(M,\frac{\Upsilon FM}{\theta_2})$ of $(FM,\Upsilon FM)$ modulo the congruence $(\theta_1,\theta_2)$, with operation $M\to \frac{\Upsilon FM}{\theta_2}$ denoted~$s$,
 \[
\begin{tikzcd}[row sep=large]
 FM \ar[r, "\upsilon"] \ar[d, "q_1"']&   \Upsilon FM\ar[d, "q_2"] \ar[dd, "k", bend  left, dashed]\\
M \ar[d, "t"']    \ar[r,   "s"] & \frac{\Upsilon FM}{\theta_2} \ar[d,  "h", dashed]\\
N \ar[r, "1_N"']&   N
\end{tikzcd}
\]
and we verify that  $s$ is the universal state of $M$. (We write $q_2\colon \Upsilon FM\to\frac{\Upsilon FM}{\theta_2}$ for the natural quotient map.) For this, let $t\colon M \to N$ be any state. The composite $tq_1\colon FM\to  N$ is a state, too. Since $\upsilon$ is universal for $FM$, there is exactly one homomorphism $k\colon \Upsilon F M\to N$ such that $k\upsilon=tq_1$. Let us show $\Ker{q_2}\seq \Ker{k}$. Since $\Ker{q_2}$ is $\theta_2$, it suffices to show that the inclusion holds for $\{(\upsilon x, \upsilon y) \mid (x,y)\in \theta_1\}$, because the latter is a generating set of $\theta_2$. Given $(\upsilon x, \upsilon y)$ with $(x,y)\in \theta_1$, from $q_1x=q_1y$ we obtain $tq_1x=tq_1y$, and thus $k\upsilon x = k\upsilon y$, that is $(\upsilon x,\upsilon y)\in \Ker{k}$, as we intended to show. By the universal property of quotients (applied to MV-algebras), there is exactly one homomorphism $h\colon  \frac{\Upsilon FM}{\theta_2} \to N$ with $hq_2=k$. This homomorphism satisfies $hs=t$. Indeed, from $hq_2=k$ and $k \upsilon = tq_1$ we obtain $hq_2\upsilon=tq_1$, and therefore $hsq_1=tq_1$; but $q_1$ is epic, whence $hs=t$. Finally, if $h'$ is any homomorphism  that satisfies $h's=t$, then $h'sq_1=tq_1$ and so $h'q_2\upsilon=tq_1$; applying the uniqueness property of $k$, we infer $h'q_2=k$, and applying that of $h$ we conclude $h=h'$. This completes the proof that $s$ is universal for $M$. 
\end{proof}
Corollary \ref{c:existenceuniv} entails at once by elementary category theory (see e.g.\ \cite[Theorem 2 in Chapter IV]{Mac1998}) that the faithful, non-full inclusion functor
\[
|-|\colon\MValg \longrightarrow\MValgs
\] 
 has a left adjoint
\begin{align}\label{eq:upsilon}
\Upsilon \colon \MValgs \longrightarrow \MValg\,.
\end{align}
 Thus, $\Upsilon M$ is the MV-algebra freely generated by $|M|$.
\begin{remark}The faithful, non-full inclusion of $\Au$ into the category of unital partially ordered Abelian groups (with morphisms the unital order-preserving group homomorphisms)  is proved to have a  left adjoint in \cite[Appendice A.2]{BigardKeimelWolfenstein77}; the argument there is for the non-unital case, but is easily adapted. Thus, the unital Abelian $\ell$-group freely generated by any unital partially ordered Abelian group exists. As a special case of this, one has that  \emph{the unital Abelian $\ell$-group freely generated by a unital partially ordered Abelian group that happens to be lattice-ordered exists}. Upon applying the results in Section \ref{sec:Gamma} to translate into the language of ordered groups, Corollary \ref{c:existenceuniv} provides an alternative proof of this result that is streamlined by the use of two-sorted algebraic theories. For clarity, we mention that  Bigard, Keimel, and Wolfenstein in  \cite{BigardKeimelWolfenstein77} distinguish between ``universal'' and free  $\ell$-groups: the former are in fact what we call ``free'', as is now standard; the latter have the further property that the universal arrow  is an order embedding. Our own usage of ``universal'' for the components of the unit of the adjunction  $\Upsilon\dashv|-|$ is meant as mere emphasis, in view of the probabilistic meaning of the construction.
\end{remark}

General free algebras in $\P$ reduce to universal states and  coproducts in $\MValg$; see \cite{Mundici88JA} for the latter. We also say ``sum'' for ``coproduct''. We write $+$ to denote binary sums in $\MValg$. If $A\to A+B \leftarrow B$ is a coproduct, we call the two arrows the coproduct injections (with no implication about their injectivity as functions), and we often denote them ${\rm in}_1$ and ${\rm in}_2$, respectively. 
\begin{theorem}\label{t:freesums}
For any two-sorted set $(S_{1},S_{2})$, the state
\begin{align}\label{eq:freesums}
(F S_{1},\Upsilon F S_1 +F S_{2} )
\end{align}
in $\P$ with operation $s$ equal to ${\rm in}_{1} \upsilon_{FS_1}$, where $\upsilon_{FS_1}$ is the universal state of $FS_1$, and ${\rm in}_{1}\colon \Upsilon FS_1\to \Upsilon F S_1 + F S_{2}$ is the first coproduct injection, is the free state generated by $S\coloneqq(S_1,S_2)$, the component of the unit at $S$ being $\eta_S\coloneqq (\iota_{S_1}, {\rm in_2}\iota_{S_2} )$ with $\iota_{S_i}$ as in \eqref{eq:freeMVunit}, $i=1,2$, and ${\rm in_2}\colon FS_2\to \Upsilon F S_1 +F S_{2}$ the second coproduct injection.

\end{theorem}
\begin{proof}We consider a two-sorted function $(f_1,f_2)\colon (S_1,S_2)\to (M,N)$. With reference to the diagram below,
\[
    \begin{tikzcd}
     S_{1} \ar[dr, "f_1"']  \ar[r, "\iota_{S_1}"] & F S_{1} \ar[dr] \ar[d, "h_1"]  \ar[r, "\upsilon_{FS_1}"] & \Upsilon F S_{1} \ar[d] \ar[r, "{\rm in}_1"] & \Upsilon F S_1 +F S_{2} \ar[dl, "h_2"']  \\
     & M \ar[r] & N & \ar[l] F S_{2} \ar[u, "{\rm in}_2"'] \\
     &          &   & S_2 \ar[u, "\iota_{S_2}"'] \ar[ul, "f_2"]
    \end{tikzcd}
\]
we have:
\begin{itemize}
\item Exactly one homomorphism  $h_1\colon FS_1\to M$ making the upper left triangle commute, by the freeness of $FS_1$;
\item the state $FS_1\to N$ given by the composition $FS_1\to M\to N$;
\item exactly one homomorphism $\Upsilon FS_1\to N$ making at its immediate left commute;
\item exactly one homomorphism  $FS_2\to N$ making the lower right triangle commute, by the freeness of $FS_2$;
\item and therefore, by the universal property of coproducts, there is exactly one homomorphism $h_2\colon \Upsilon FS_1+ FS_2\to N$ such that precomposing it with the coproduct injections ${\rm in}_1$ and ${\rm in_2}$ yields $\Upsilon FS_1\to N$ and $FS_2\to N$, respectively.
\end{itemize}
By construction, the pair $(h_1,h_2)$ is a morphism in $\P$ that  satisfies $(f_1,f_2)=(h_1,h_2)(\iota_{S_1}, {\rm in_2}\iota_{S_2} )$. That it is the unique such follows readily from  the diagram  above.
\end{proof}
\section{Universal states, and Choquet's affine representation}\label{s:choquet}
We recall basic facts about the theory of affine representations. For background information and  references see \cite[Chapters 5--7]{Goodearl86}. We formulate the results in~the~language of ordered groups, as is traditional; they may be translated for MV-algebras via the equivalence in Theorem \ref{thm:Mundiciequiv}.

For a  unital Abelian $\ell$-group $G$, set 
\[
\St{G}\coloneqq\left\{s\colon G\to\R\mid s \text{ is a real-valued state}\right\}\subseteq\R^G\,.
\]
 Equip $\R$ with its Euclidean topology, $\R^G$ with the product  topology, and $\St{G}$ with the subspace topology. Then $\St{G}$, the \emph{state space of $G$}, is a compact  Hausdorff space, which is moreover a convex set in the vector space $\R^G$. For every $a\in G$ we consider the function
\begin{align}\label{eq:transf}
\hat{a}\colon \St{G}&\longrightarrow\R\\\nonumber
s &\longmapsto s(a).
\end{align}
The function $\hat{a}$ is continuous and affine; we  write $\Aff{(\St{G})}$ for the set of all continuous affine maps $\St{G}\to\R$. Then \cite[Theorem 11.21]{Goodearl86} says that $\Aff{(\St{G})}$ is a~unital  Abelian $\ell$-group under pointwise addition and order, with unit the function constantly equal to $1$. The  map
\begin{align}
e_G\colon G&\xrightarrow{\ \ \ \ }\Aff{(\St{G})}\label{eq:affrep}\\ \nonumber
a&\xmapsto{\ \ \ \ } \hat{a},
\end{align}
induced  by \eqref{eq:transf} is  a state $G\to \Aff{(\St G)}$. It is an isomorphism onto its range  precisely when $G$ is Archimedean \cite[Theorem 7.7]{Goodearl86}. Recall that $G$ is \emph{Archimedean} if for $a,b\in G$, $na\leq b$ for all positive integers $n$ implies $a\leq 0$. In any case, even when $G$ fails to be Archimedean, we call \eqref{eq:affrep} the \emph{affine representation} of $G$. 

For any compact Hausdorff space $X$, write $\C{(X)}$ for the unital Abelian $\ell$-group of continuous real-valued functions on $X$, operations being defined pointwise; the function   constantly equal to $1$ is the  unit.  For every unital Abelian $\ell$-group $G$, the inclusion $\Aff{(\St{G})}\subseteq \C{(\St{G})}$ preserves the unit, the group structure, and the partial order. In general, however, it fails to preserve the lattice structure. Let us therefore consider the sublattice-subgroup $\widehat{G}$ of $\C{(\St{G})}$ generated by the image of $G$ under the affine representation map $e_G$ in  \eqref{eq:affrep}. Modifying the codomain of the~ affine representation accordingly, but retaining the same notation, we obtain the state
\begin{align}
e_G\colon G&\xrightarrow{\ \ \ \ }\widehat{G}\label{eq:paffrep}\\ \nonumber
a&\xmapsto{\ \ \ \ } \hat{a}, 
\end{align}
which we call the \emph{extended affine representation} of $G$. 

Reformulating the notion of universal state of MV-algebras, a state
\[
\upsilon_G\colon G\longrightarrow \Upsilon G
\]
of the unital Abelian $\ell$-groups $G$ will be called universal (for $G$) if for each state $s\colon G\to H$ there is exactly one unital $\ell$-homomorphism $h\colon \Upsilon G\to H$ satisfying $h\upsilon_G=s$. We will relate the extended affine representation (\ref{eq:paffrep}) of $G$ with the codomain $\Upsilon G$ of a universal state. In light of the universal property of $v_G$ there is exactly one comparison unital $\ell$-homomorphism 

\begin{equation}\label{eq:comparison}
  q_G\colon \Upsilon G\longrightarrow \widehat{G}
  \end{equation}
  that satisfies 
  \begin{equation}\label{eq:comparisoneq}
  q_G v_G=e_G\,.
  \end{equation}

\begin{proposition}\label{pro:semisimplecodomain}
Let $G$ be a unital Abelian $\ell$-group. The homomorphism $q_G
$ in~\eqref{eq:comparison} is an isomorphism if, and only if, $\Upsilon G$ is Archimedean.
\end{proposition}
\begin{proof}First observe that $q_G$ is always surjective. Indeed, $v_G[G]$ and $e_G[G]$ generate $\Upsilon G$ and $\widehat{G}$ as  $\ell$-groups, respectively. Therefore,  by \eqref{eq:comparisoneq}, the map $q_G$ throws a~generating set of $\Upsilon G$ onto a generating set of $\widehat{G}$; hence it is surjective. Further,
  recall that the \emph{radical ideal} of a unital Abelian $\ell$-group $H$ is defined as
  \[
  \Rad{H}\coloneqq\bigcap\left\{\ker{h}\mid h\colon H\to\R  \text{ is a homomorphism}\right\},
  \]
  and that $H$ is Archimedean if, and only if, $\Rad{H}=\{0\}$. It now suffices to prove that $\ker{q_G}=\Rad{\Upsilon G}$. Since 
  $\widehat{G}$ is Archimedean by construction, and unital $\ell$-homomorphisms such as $q_G$ preserve radical ideals, the inclusion $\Rad{\Upsilon G}\subseteq\ker{q_G}$ is clear. For the converse inclusion, suppose by contraposition that $x\not\in\Rad{\Upsilon G}$, with the intent of showing $q_Gx\neq 0$. By the hypothesis there is a unital $\ell$-homomorphism $h\colon \Upsilon G\to \R$ such that $hx\neq 0$. Thus we have a real-valued state 
  \[
  s\coloneqq h v_G\colon G\longrightarrow \R.
  \]
  Evaluation of elements of $\widehat{G}$ at $s$ produces a homomorphism 
  \begin{align*}
  {\rm ev}_{s}\colon\widehat{G}&\longrightarrow \R\\
  f&\longmapsto fs\in\R
  \end{align*}
  such that 
  \begin{align}\label{eq:eval}
  {\rm ev}_{s} e_G= h v_G.
  \end{align}
  To see that \eqref{eq:eval} holds, pick $a\in G$ and compute: $({\rm ev}_s e_G)(a) =  {\rm ev}_s(e_G(a))=(e_G(a))(s)=s(a)$, where the last equality is given by the definition  \eqref{eq:transf}. Since $s= h v_G$, \eqref{eq:eval} holds.

  From (\ref{eq:comparisoneq}--\ref{eq:eval}) we deduce
  \[
  {\rm ev}_{s} q_G v_G = h v_G\,,
  \]
  which by the universal property of $v_G$ entails
  \begin{align}\label{eq:infer}
  {\rm ev}_s q_G=h.
  \end{align}
  Since $hx\neq 0$ by hypothesis, from \eqref{eq:infer} we infer $q_Gx\neq 0$, as was to be shown.
  \end{proof}
\begin{remark}\label{r:Bernau}The question of when the  Abelian $\ell$-group freely generated by an Archimedean partially ordered Abelian group is itself Archimedean has long had an important place in the theory of ordered groups.  Bernau \cite{Bernau1969} gave an example of an Archimedean partially ordered Abelian group, which moreover happens to be lattice-ordered, such that the Abelian $\ell$-group it freely generates  fails to be Archimedean. Further, Bernau obtained in \cite[Theorem 4.3]{Bernau1969} a necessary and sufficient condition for the  Abelian $\ell$-group freely generated by a partially ordered Abelian group to be Archimedean; the condition, which we do not reproduce here due to its length, is known as the \emph{uniform Archimedean property} of a partially ordered group. Via the results in Section \ref{sec:Gamma},  Bernau's example entails that the universal state of a semisimple MV-algebra may have a non-semisimple codomain. However,  we shall see in the next section that this  cannot happen when the semisimple MV-algebra in question is locally finite.
\end{remark}
\section{The universal state of a locally finite MV-algebra}\label{s:univlocfin}
An MV-algebra  is \emph{locally finite} if each of its finitely generated subalgebras is finite. By general algebraic considerations, any algebraic structure is the direct union of its finitely generated subalgebras \cite[Theorem 2.7]{Jacobson_II_89}. Thus, every locally finite MV-algebra is the direct union of its finite subalgebras. We are going to prove:
\begin{theorem}\label{t:univlocfinisreal}Let $\upsilon_M\colon M\to \Upsilon M$ be the universal state of a locally finite MV-algebra $M$ (cf.\ Corollary \ref{c:existenceuniv}). Then $\Upsilon M$ is semisimple. 
\end{theorem}
The proof of this theorem will require a number of lemmas. For each integer $i\geq 1$, let $M_i\coloneqq \{0,\frac {1}{i},\dots, \frac{i-1}{i},1\}$ be the finite totally ordered MV-algebra of cardinality $i+1$. By \cite[Proposition 3.6.5]{CignoliOttavianoMundici00}, a finite  MV-algebra $A$  is isomorphic to a product $M_{k_1}\times\dots\times M_{k_n}$, for uniquely determined integers $k_1,\dots,k_n \geq 1$, $n\geq 0$. (When $n=0$, $A$ is the terminal MV-algebra $\{0=1\}$.) For the rest of this section we use $A$ to denote such a finite MV-algebra.
\begin{lemma}\label{lem:ext}
    Let $A=M_{k_1}\times\dots\times M_{k_n}$ be a finite MV-algebra and $N$ be any MV-algebra. Write $\{a_1,\ldots,a_n\}$ for the atoms of $A$.
    \begin{enumerate}
    \item Any state $s\colon A\to N$ is uniquely determined by its values at the atoms of~$A$.
    \item A function $s\colon \{a_1,\dots ,a_n\}\to N$ has an extension to a state $A\to N$ if, and only if, $k_1s(a_1) + \dots + k_ns(a_n)=1$, and in that case the extension is unique.
    \end{enumerate}
  \end{lemma}
  \begin{proof}The unital Abelian $\ell$-group $\Xi A$ is  the simplicial group $\Z^n$ with unit the element $(k_1,\ldots,k_n)$, and the atoms $a_i$ are the standard basis elements of $\Z^n$; from this item 1 follows at once. As for item 2, the left-to-right implication follows directly from the definition of state. Conversely, assume $k_1s(a_1) + \dots + k_ns(a_n)=1$. For any  $a\in \Gamma\Xi A\subseteq \Xi A= \Z^n$, we can write $a=\sum_{i=1}^nc_ia_i$ for uniquely determined integers $c_i$. Setting $s'(a)\coloneqq \sum_{i=1}^n c_is(a_i)$, one  verifies  that $s'$ is a state, and then, by item 1, $s'$ is  the unique state extending $s$.
  \end{proof}

 Set $\mathbf{k}\coloneqq(k_1,\ldots,k_n)$, for short. If $Fn$ is the  MV-algebra freely generated by the set $\{x_1,\ldots,x_n\}$, there is a finitely 
generated (hence principal, by \cite[Lemma 1.2.1]{CignoliOttavianoMundici00}) ideal $U(\mathbf{k})$ of $Fn$ determined by the partition-of-unity relation
\begin{align}\label{eq:partitionunity}
\sum_{i=1}^nk_ix_i=1,
\end{align}
where addition is interpreted in the unital Abelian $\ell$-group $\Xi Fn$. In more detail, there is a term $\sigma(x_1,\ldots,x_n)$ in the language of MV-algebras such that, for any MV-algebra $M$, and for any elements $b_1,\ldots,b_n\in M$,     $\sigma(b_1,\dots,b_n)=0$ holds in $M$ if, and only if, $\sum_{i=1}^nk_ib_i=1$ holds in $\Xi M$.  Thus,  $U(\mathbf{k})$ is  the ideal of $Fn$ generated  by $\sigma(x_1,\ldots,x_n)$, when the latter is regarded as an element of $Fn$ in the usual way. For an explicit computation of $\sigma$ the interested reader can consult \cite{Mundici2000reasoning}. We write $S_{\mathbf{k}}$ for the quotient algebra $F n/U(\mathbf{k})$. By Lemma \ref{lem:ext}, the function
\[
a_i\longmapsto x_i,\qquad i=1,\ldots,n,
\]
is extended by exactly one state 
\begin{align}\label{eq:univstatefin}
\upsilon_A\colon A \longrightarrow S_{\mathbf{k}},
\end{align}
where $A=M_{k_1}\times\dots\times M_{k_n}$.
\begin{lemma}\label{l:geomrep} The state \eqref{eq:univstatefin} is the universal state of the MV-algebra $A=M_{k_1}\times\dots\times M_{k_n}$.
\end{lemma}
\begin{proof}If $s\colon A \to N$ is any state, and $\{a_1,\ldots,a_n\}$ is the set of atoms of $A$, consider the assignment $x_i\mapsto s(a_i)$, $i=1,\ldots, n$. We have $\sum_{i=1}^n k_is(a_i)=1$ in $N$, because $s$ is a state. Then, by the definition of  $S_{\mathbf{k}}$ and the universal property of homomorphic images in varieties, this assignment has exactly one extension to a homomorphism $h\colon S_{\mathbf{k}}\to N$. For each $a_i\in A$ we have $(h\upsilon_A)a_i=s(a_i)$, which entails $h\upsilon_A=s$ because $\{a_1,\ldots, a_n\}$ is a $\Z$-module basis of $\Xi A=\Z^n$. Such $h$ is unique, because if $h'\colon S_{\mathbf{k}}\to N$ satisfies $h'\upsilon_A=s$ then $h$ and $h'$ must agree on the generating set $\{x_1,\ldots,x_n\}$ of $S_{\mathbf{k}}$: indeed, $h'\upsilon_A=h\upsilon_A$ entails $h'x_i=(h'\upsilon_A)a_i=(h\upsilon_A)a_i=hx_i$, $i=1,\ldots,n$.
\end{proof}  
For the proof of the last lemma we need, Lemma \ref{l:inj} below, we  shall apply geometric techniques. We use the integer $d\geq 0$ to denote the dimension of the
real vector space $\R^d$.  A function $f \colon \R^d \to \R$ is \emph{PL} (for \emph{piecewise linear}) if it is continuous
with respect to the
Euclidean topology on $\R^d$ and $\R$, and there is a finite set of 
affine functions $l_1,\ldots,l_u\colon \R^d\to\R$ such that for each $x \in \R^d$ one has $f(x)=l_i(x)$
for some choice of $i=1,\ldots,u$. (We note in passing that the terminology ``piecewise linear'' is traditional, even though ``piecewise affine'' would be, strictly speaking, more appropriate.) If moreover, each $l_i\colon \R^d\to\R$ can be chosen so as to restrict to a function $\Z^d\to\Z$, then $f$ is a \emph{$\Z$-map}. (This  terminology comes from \cite{Mundici11}.) In coordinates, this is equivalent to asking that $l_i$ can be written as a linear polynomial with integer
coefficients.
For an integer $d'\geq 0$, a function $\lambda =(\lambda_1,\ldots,\lambda_{d'})\colon \R^d \to \R^{d'}$ is a
\emph{PL map} (respectively, \textit{a $\Z$-map}) if 
each one of its scalar components $\lambda_j\colon \R^d\to \R$ is a PL function ($\Z$-map). One defines PL  maps ($\Z$-maps) $A \to B$ for arbitrary
subsets $A\subseteq \R^d$, $B\subseteq \R^{d'}$ as the restriction and
co-restriction of PL maps  ($\Z$-maps). 

A \emph{convex combination} of a finite set of vectors $v_1,\ldots,v_u \in \R^d$
is any vector of the form
$\lambda_1v_1+\cdots+\lambda_uv_u$, for non-negative real numbers $\lambda_i
\geq 0$ satisfying $\sum_{i=1}^u\lambda_i=1$.
If $S\subseteq \R^d$ is any subset, we let $\conv{S}$ denote the \emph{convex
hull} of $S$, i.e.\ the collection of all convex combinations
of finite sets of vectors $v_1,\ldots,v_u \in S$. A \emph{polytope}  is any
subset of $\R^d$ of the form
$\conv{S}$, for some finite $S\subseteq \R^d$, and a (\emph{compact})
\emph{polyhedron}  is a  union of finitely many polytopes
in $\R^d$. A polytope is \emph{rational} if it may be written in the form
$\conv{S}$ for some finite set $S\subseteq \Q^d\subseteq \R^d$ of vectors with
rational coordinates. Similarly, a polyhedron is \emph{rational}
if it may be written as a union of finitely many rational polytopes.  
It is clear that the composition of $\Z$-maps is again a $\Z$-map. Rational polyhedra and $\Z$-maps thus form a category, which we denote $\PolyZ$. This is a non-full subcategory of the classical compact polyhedral category $\Poly$ whose objects are polyhedra and whose morphisms are PL maps.
\begin{remark}\label{r:reductiontocubes}The full subcategory of $\PolyZ$ whose objects are rational polyhedra lying in unit cubes $[0,1]^d$, as  $d$ ranges over all non-negative integers, is equivalent to  $\PolyZ$, see \cite[Claim 3.5]{marra_spada}. An analogous remark applies to $\Poly$. We shall make use of these facts whenever convenient, without further warning.\end{remark}
Given a subset  $S\seq \R^d$, we write 
    ${\MM}S$ for the set of all $\Z$-maps $S \to [0,1]$. Then ${\MM}S$ inherits from  $[0,1]$ the structure of an MV-algebra, upon defining operations pointwise; in other words, ${\MM}S$ is a subalgebra of the MV-algebra $\C{(S,[0,1])}$ of all continuous functions $S\to [0,1]$. 
It can be proved that if $X\seq \R^d$ is a rational polyhedron then ${\MM}X$ is finitely presentable; see e.g.\ \cite[Theorem 6.3]{Mundici11}, where the result is stated for $X\seq [0,1]^d$---the case $X\seq \R^d$  is then a consequence of Remark \ref{r:reductiontocubes}. Following tradition, from now on we say `finitely presented' instead of `finitely presentable', even when the latter would be the proper expression. By $\MVfp$ we denote the full subcategory of $\MV$ whose objects are finitely presented MV-algebras.
Further, if $P\subseteq \R^d$ and $Q\subseteq \R^{d'}$ are rational polyhedra for some integers $d,d'\geq 0$, a $\Z$-map $\lambda \colon P \to Q$
induces a function
\[
 {\MM}\lambda \colon {\MM}Q \xrightarrow{ \ \ \ \  } {\MM}P
\]
given by
\[
 f \in {\MM}Q \xmapsto{\ {\MM}\lambda \ } f\lambda \in {\MM}P.
\]
It can be shown that ${\MM}\lambda$ is a homomorphism of MV-algebras, see e.g.\ \cite[Lemma 3.3]{marra_spada}. We thereby obtain a functor
\begin{align}\label{eq:MVdual}
\MM\colon \PolyZ\xrightarrow{\ \ \ \ } \MVfpop.
\end{align}
To define a functor
\begin{align}\label{eq:MVdualbis}
\Max\colon \MVfpop \xrightarrow{\ \ \ \ }  \PolyZ
\end{align} in the other direction,  let us assume that a specific finite presentation of an MV-algebra is given. That is, we are given a finite  set $R\coloneqq\{\rho_i\in Fn\mid i=1,\ldots,m\}$ of elements of $Fn$, the free MV-algebra generated by $n$ elements $x_1,\ldots,x_n$, and we consider the finitely presented quotient $Fn/\langle R \rangle$, where $\langle R \rangle$ denotes the ideal of $Fn$ generated by $R$. Since every finitely generated ideal of any MV-algebra is principal by \cite[Lemma 1.2.1]{CignoliOttavianoMundici00}, we may safely assume  that $\langle R \rangle$ is generated by the single element $\rho$.
Writing $\pi_i\colon [0,1]^n\to [0,1]$ for the $i^{\rm th}$ projection function, the assignment $x_i\mapsto \pi_i$, $i=1,\ldots,n$, extends to exactly one homomorphism $Fn\to \MM [0,1]^n$ by the universal property of the free algebra, and this homomorphism is an isomorphism by \cite[Theorem 9.1.5]{CignoliOttavianoMundici00}. Thus, ${\MM}[0,1]^n$ is the MV-algebra freely generated by the projection functions. Under this isomorphism, $\rho$ corresponds to an  element $\rho'$ of ${\MM}[0,1]^n$, and the assignment $x_i\mapsto \pi_i$, $i=1,\ldots,n$, extends to exactly one isomorphism  $Fn/\langle \rho \rangle \to {\MM}[0,1]^n/\langle \rho'\rangle$. Then we set $\Max{Fn/\langle \rho \rangle}
\coloneqq \rho'^{-1}\{0\}$. Here, the zero set $\rho'^{-1}\{0\}$ of the $\Z$-map $\rho'$ is a rational polyhedron in $[0,1]^n$. When regarded  as a topological space, by \cite[Theorem 3.4.3]{CignoliOttavianoMundici00} the polyhedron $\rho'^{-1}\{0\}$ is homeomorphic to the maximal spectral space of the MV-algebra $Fn/\langle \rho \rangle$, whence the name of  the functor \eqref{eq:MVdualbis}. In turn, the maximal ideals of any MV-algebra $M$ are in bijection with the kernels of the homomorphisms $M\to [0,1]$ (\cite[Theorem 4.16]{Mundici11}). Therefore  the points of $\rho'^{-1}\{0\}$ are in bijection with (the kernels of) the homomorphisms $Fn/\langle \rho \rangle\to [0,1]$. To define $\Max$ on arrows, then, consider a homomorphism $h\colon Fn/\langle\rho\rangle\to Fm/\langle\tau\rangle$. Precomposing with $h$, a morphism $Fm/\langle\tau\rangle\to [0,1]$ is sent to a morphism $Fn/\langle\rho\rangle$, so that we obtain a mapping $\Max h\colon\tau'^{-1}\{0\}\to\rho'^{-1}\{0\}$ which can be proved to be a $\Z$-map. This is the definition of the functor $\Max$ for finitely presented MV-algebras; for the general case of a finitely presentable algebra, to complete the definition one chooses a finite presentation for each algebra. 
\begin{theorem}[Duality theorem for finitely presented MV-algebras]\label{t:duality}The  functors $\MM\colon \PolyZ\xrightarrow{\ \ } \MVfpop$ and $\Max\colon \MVfpop \xrightarrow{\ \ }  \PolyZ$  form an equivalence of categories.
\end{theorem}
\begin{proof}Two different proofs of this result are given in \cite{marra_spada_sl} and \cite{marra_spada}.
\end{proof}
\begin{remark}We will apply Theorem \ref{t:duality}  in the sequel without explicit reference to its numbered statement, freely using  such
expressions as ``the dual rational polyhedron $X$ of the finitely presented MV-algebra $A$''.
\end{remark}
 If $x \in \Q^d$, there is a unique way to write out $x$ in
coordinates as 
\[
 x = \left(\frac{p_1}{q_1},\ldots, \frac{p_d}{q_d}\right)
 \]
with $p_i,q_i \in\Z \ , \ q_i > 0$, $p_i$ and $q_i$  relatively prime for each $i=1,\ldots, d$.
The positive integer 
\[
\den{x}\coloneqq{\rm lcm}\,{\{q_1,q_2,\ldots,q_d\}}
\]
is the \emph{denominator} of $x$. For a non-negative integer $t$, a \emph{$t$-dimensional simplex} (or just \emph{$t$-simplex}) in $\R^d$ is the convex hull of $t+1$ affinely independent points in $\R^d$, called its \emph{vertices}; the vertices of a simplex are uniquely determined. A~\emph{face} of a~simplex $\sigma$ is the convex hull of a nonempty subset of the vertices of $\sigma$, and as such it is itself a simplex. The \emph{relative interior} of $\sigma$ is the set of points expressible as convex combinations of its vertices with strictly positive coefficients.
A simplex is \emph{rational} if its vertices are. The following notion is fundamental to the arithmetic geometry of $\PolyZ$: A simplex $\sigma\seq \R^d$ is \emph{regular} if whenever $x$ is a rational point lying in the relative interior of some face $\tau$ of $\sigma$ with vertices $v_1,\ldots,v_l$, then $\den{x}\geq \sum_{i=1}^l\den{v_i}$. See \cite[Lemma 2.7]{Mundici11}. Regular simplices are also called \emph{unimodular} in the literature.

Since the finite MV-algebra $A=M_{k_1}\times \cdots \times M_{k_n}$ is finitely presented, it has a dual polyhedron, which by duality must be the sum of the duals of $M_{k_i}$, $i=1,\ldots,{n}$.  The dual of each $M_{k_i}$ is immediately seen to be a single rational point $p$ (in some $\R^d$) such that $\den{p}=k_i$. Thus, the dual of $A$ is any finite set of points $\{p_1,\ldots,p_n\}$ such that $\den{p_i}=k_i$, $i=1,\ldots,n$. A specific coordinatisation of such a rational polyhedron is obtained as follows. Let $\{e_1,\ldots, e_n\}$ be the standard basis vectors in $\R^n$. Then $\den{\frac{e_i}{k_i}}=k_i$. Thus, $\{\frac{e_1}{k_1},\ldots, \frac{e_n}{k_n}\}$ is the rational polyhedron dual to $A$. The convex hull
\[
\Delta_{\mathbf{k}}\coloneqq\conv{\left\{\frac{e_1}{k_1},\ldots, \frac{e_n}{k_n}\right\}}
\]
is a regular simplex by direct inspection. For each $i=1,\ldots, n$, we write $\pi_i^{\mathbf{k}}\colon \Delta_{\mathbf{k}}\to [0,1]$ for the restriction of the projection function $\pi_i\colon[0,1]^n\to [0,1]$. Evidently, $\pi_i^\mathbf{k}\in{\MM}\Delta_{\mathbf{k}}$. Observe that
\begin{align}\label{eq:partitionofunitypi}
\sum_{i=1}^n k_i\pi_i^\mathbf{k}=1.
\end{align}
\begin{lemma}\label{l:dualofrefsim}For any $\mathbf{k}=(k_1,\ldots,k_n)$, the dual polyhedron of the finitely presented  MV-algebra $S_{\mathbf{k}}$ is $\Delta_{\mathbf{k}}$. In more detail, writing $x_i^\mathbf{k}$ for the generators provided by the presentation of $S_{\mathbf{k}}$,  the assignment $x_i^\mathbf{k}\mapsto \pi_i^\mathbf{k}$, $i=1,\ldots,n$, extends to exactly one homomorphism $S_{\mathbf{k}}\to {\MM}\Delta_{\mathbf{k}}$, and that homomorphism is an isomorphism.
\end{lemma}
\begin{proof}Applying the definition of the functor $\Max$ in \eqref{eq:MVdualbis}, the polyhedron $\Max S_{\mathbf{k}}$ dual to $S_{\mathbf{k}}$ is the solution set in $[0,1]^n$ of the equation $\sum_{i=1}^n k_i\pi_i=1$ that presents~$S_{\mathbf{k}}$ (or equivalently, it is the zero set of $\sigma(\pi_1,\ldots,\pi_n)$, where $\sigma$ is as in the definition of~$S_{\mathbf{k}}$). Computation confirms that this solution set is   $\Delta_{\mathbf{k}}$; and the remaining part of the statement is a direct consequence of  the duality Theorem \ref{t:duality}.
\end{proof}
\begin{corollary}\label{c:univgeom}There is exactly one state 
\begin{align}
\upsilon_A\colon A \longrightarrow {\MM}\Delta_{\mathbf{k}}
\end{align}
that extends the assignment
\[
a_i\longmapsto \pi_i,\qquad i=1,\ldots,n,
\]
and this is the universal state of the MV-algebra $A=M_{k_1}\times\cdots\times M_{k_n}$.
\end{corollary}
\begin{proof}The existence and uniqueness of the extension follows from Lemma \ref{lem:ext}, upon recalling \eqref{eq:partitionofunitypi}. That the state in question is naturally isomorphic to $\upsilon_A$ in Lemma \ref{l:geomrep} follows at once from the isomorphism in Lemma \ref{l:dualofrefsim}.
\end{proof}
We record a well-known elementary property of $\Z$-maps for which we could not locate a reference (however, see the related \cite[Lemma 3.7]{Mundici11}).
\begin{lemma}\label{l:zmapsdivide}Let $X\subseteq \R^d$ and $Y\subseteq \R^{d'}$ be rational polyhedra, for integers $d,d' \geq 0$, and let $f\colon X\to Y$ be a $\Z$-map.
For each $x \in X\cap \Q^d$, the number $fx$ is rational and $\den{fx}$ divides $\den{x}$; in symbols, $\den{fx}\mid \den{x}$.
\end{lemma}
\begin{proof}The $\Z$-map $f\colon X \to Y$ may be written in vectorial form as $(f_1, \ldots, f_{d'})$, with each $f_i\colon X \to \R$ a $\Z$-map. In the rest of this proof, assume all rational numbers are expressed in reduced form. Pick a point $x=(\frac{p_i}{q_i})\in X\cap \Q^d$,  and regard $x$ as a  column vector. By definition, each $f_i$ agrees   locally at the point $x$ with an affine function $\R^d \to \R$ with integer coefficients. Thus, there is a column vector $c=(c_i)\in\Z^{d'}$ together with a $(d'\times d)$ matrix $M=(z_{ij})$ with integer entries such that 
\begin{align}\label{eq:matrix}
fx=Mx+c.
\end{align}
This shows  that $fx$ is rational if $x$ is. Further, let us write $fx=(\frac{a_i}{b_i})\in\Q^{d'}$. By \eqref{eq:matrix} we have  $\frac{a_i}{b_i}=\sum_{j=1}^{d}z_{ij}\frac{p_j}{q_j} + c_i$, so that $b_i \mid \den{x}$, and therefore ${\rm lcm}\,\{b_i\}_{i=1}^{d'}=\den{fx}\mid \den{x}$.
\end{proof}
We can now prove the promised lemma.
\begin{lemma}\label{l:inj} If $h\colon A\to B$ is any injective homomorphism between finite MV-algebras, $\Upsilon h\colon\Upsilon A \to \Upsilon B$ is injective. 
\end{lemma}
\begin{proof}Let $\{e_1/k_1,\ldots,e_n/k_n\}\seq \R^n$ be the dual finite rational polyhedron of $A$, with $\den{e_i/k_i}=k_i$, $i=1,\ldots,n$. Similarly, let $\{e_1/t_1,\ldots,e_m/t_m\}\seq \R^m$ be the dual of $B$, with $\den{e_i/t_i}=t_i$, $i=1,\ldots,m$. Let us write  \[h^*\coloneqq\Max h\colon \{e_1/t_1,\ldots,e_m/t_m\} \to \{e_1/k_1,\ldots,e_n/k_n\}\]
for the $\Z$-map dual to $h$.  By Lemma \ref{l:zmapsdivide} we  have $k_i\mid t_j$ whenever $h^*(e_j/t_j)=e_i/k_i$. Moreover, the $\Z$-map $h^*$ has exactly one extension to an affine map $\alpha\colon \Delta_{\mathbf{t}}\to\Delta_{\mathbf{k}}$, where $\mathbf{t}\coloneqq (t_i)_{i=1}^m$, by the elementary properties of affine functions and simplices. The computation in \cite[Lemma 3.7]{Mundici11} (or else, which is essentially the same, direct computation of the matrix form of $\alpha$ with respect to standard bases) confirms that $\alpha$ is a $\Z$-map. Since $h$ is injective, it is a monomorphism, and thus $h^*$ is an~epimorphism by duality. But then a straightforward verification shows that $h^*$ is a~surjective function. This entails at once that its affine extension $\alpha$ is surjective, too. For any two continuous functions $f,g\colon\Delta_{\mathbf{k}}\to \R$ with $f\neq g$ we then have $f\alpha\neq g\alpha$, by the surjectivity of $\alpha$. In particular, this means that the homomorphism
\[
{\MM}\alpha\colon {\MM}\Delta_{\mathbf{k}}\longrightarrow {\MM}\Delta_{\mathbf{t}},
\]
which by the definition of the functor $\MM$ carries a $\Z$-map $f\colon\Delta_{\mathbf{k}}\to [0,1]$ to the $\Z$-map $f\alpha\colon \Delta_{\mathbf{t}}\to [0,1]$, 
is injective. The proof is now completed by showing that $\Upsilon h$ is ${\MM}\alpha$ to within a natural isomorphism. That is, in light of Corollary \ref{c:univgeom}, we need to check that the  square
\[
\begin{tikzcd}
A \ar[d, "\upsilon_A"']    \ar[r,   "h"] & B \ar[d,  "\upsilon_B"]\\
{\MM}\Delta_{\mathbf{k}} \ar[r, "{\MM\alpha}"']&   {\MM}\Delta_{\mathbf{t}}
\end{tikzcd}
\]
commutes, which amounts to a straightforward application of the definitions which we leave to the reader.
\end{proof}
Finally:
\begin{proof}[Proof of Theorem \ref{t:univlocfinisreal}]Write $\Sigma$ for the directed partially ordered set of all finite subalgebras of $M$, so that $M=\bigcup_{A\in\Sigma}A$ (direct union). The functor $\Upsilon$ in \eqref{eq:upsilon} is left adjoint and thus preserves colimits  \cite[Theorem 1 in Chapter V.5]{Mac1998}. Moreover, by Lemma \ref{l:inj}, each inclusion $A\to B$ with $A,B\in\Sigma$ is sent by $\Upsilon$ to an injective homomorphism. Together with  Corollary \ref{c:univgeom}, this says that $\Upsilon M$ is a directed colimit of semisimple MV-algebras with injective transition homomorphisms.  Since the radical ideal of   semisimple MV-algebras is trivial,  the radical of the directed colimit $\Upsilon M$ is trivial, too, by \cite[Proposition 3.6.4]{CignoliOttavianoMundici00}. Hence, $\Upsilon M$ is semisimple.
\end{proof}
\begin{corollary}\label{c:choquet}For any locally finite MV-algebra $M$, the extended affine representation of $M$ (cf.\ Section \ref{s:choquet}) is the universal state of $M$. This happens, in particular, for all Boolean algebras.
 \end{corollary}
 \begin{proof}Theorem \ref{t:univlocfinisreal} and Proposition \ref{pro:semisimplecodomain}. The last assertion holds because of the standard fact that finitely generated Boolean algebras are finite. 
 \end{proof}
\section{Further research}
The present paper is part of a nascent programme aimed at exploring universal constructions in probability theory. While we refrain from sketching that programme here, we do   point out some connections between the line of research pursued in this paper, and many-valued logic.  \emph{Fuzzy Probability Logic} FP(\L) over infinite-valued \L ukasiewicz logic was developed by H\'{a}jek  \cite[Chapter 8.4]{Hajek98}, and later extended by Cintula and Noguera to a two-tier modal logic aimed at modelling uncertainty  \cite{CintulaNoguera14}. The logic FP(\L) formalises reasoning about properties of states, similarly to probabilistic logics designed for reasoning about probability.  The main feature of FP(\L) is a two-level syntax. Probability assessments are represented in the language by a unary modality ‘Probably’, which can be applied to Boolean formul\ae\ only. The class of MV-algebras and states provides a possible complete semantics for FP(\L); see  \cite{CintulaNoguera14} for more details. States as two-sorted  algebras, as introduced here, may provide an equivalent multi-sorted algebraic semantics to the logic FP(\L). Details will be pursued in further research.
 
 \section*{Acknowledgments}
 Both authors are deeply grateful to an anonymous reviewer for his/her careful and competent reading of a previous version of this paper, which contained  two significant oversights in the proof of Corollary \ref{c:existenceuniv}.

 The work of Tom\'a\v{s} Kroupa has been supported from the GA\v{C}R grant project GA17-04630S and from the project RCI (CZ.02.1.01/0.0/0.0/16\_019/0000765).

\end{document}